\input ./style/arxiv-vmsta.cfg
\documentclass[numbers,compress,v1.0.1]{vmsta}
\usepackage{vtexbibtags}

\volume{7}
\issue{2}
\pubyear{2020}
\firstpage{191}
\lastpage{202}
\aid{VMSTA156}
\doi{10.15559/20-VMSTA156}
\articletype{research-article}

\startlocaldefs
\newtheorem{thm}{Theorem}
\newtheorem{lemma}{Lemma}

\theoremstyle{definition}
\newtheorem{remark}{Remark}

\ifdefined\HCode
\def\index#1{}
\fi
\allowdisplaybreaks
\endlocaldefs

\begin{document}

\begin{frontmatter}
\pretitle{Research Article}

\title{Distance from fractional Brownian motion with associated Hurst
index $0<H<1/2$ to the subspaces of Gaussian martingales involving power
integrands with an arbitrary positive exponent}

\author[a]{\inits{O.}\fnms{Oksana}~\snm{Banna}\thanksref{cor1}\ead[label=e1]{okskot@ukr.net}\orcid{0000-0002-9730-4654}}
\thankstext[type=corresp,id=cor1]{Corresponding author.}
\author[b]{\inits{F.}\fnms{Filipp}~\snm{Buryak}\ead[label=e1]{filippburyak2000@gmail.com}}
\author[b]{\inits{Yu.}\fnms{Yuliya}~\snm{Mishura}\ead[label=e2]{myus@univ.kiev.ua}\orcid{0000-0002-6877-1800}}
\address[a]{\institution{Kyiv National Taras Shevchenko University}, Faculty of Economics,
Volodymyrska~64, 01601 Kyiv, \cny{Ukraine}}
\address[b]{\institution{Kyiv National Taras Shevchenko University}, Faculty of Mechanics and Mathematics,
Volodymyrska~64, 01601 Kyiv, \cny{Ukraine}}

\runtitle{Distance from fBm to the subspaces of Gaussian martingales}


\begin{abstract}
We find the best approximation of the fractional Brownian motion with the Hurst
index $H\in (0,1/2)$ by Gaussian martingales of the form
$\int _{0}^{t} s^{\gamma }dW_{s}$, where $W$ is a Wiener process,
$\gamma >0$.
\end{abstract}
\begin{keywords}
\kwd{Fractional Brownian motion}
\kwd{martingale}
\kwd{approximation}
\end{keywords}

\begin{keywords}[MSC2010]
\kwd{60G22}
\kwd{60G44}
\end{keywords}

\received{\sday{22} \smonth{4} \syear{2020}}
\revised{\sday{6} \smonth{6} \syear{2020}}
\accepted{\sday{6} \smonth{6} \syear{2020}}
\publishedonline{\sday{23} \smonth{6} \syear{2020}}

\end{frontmatter}
%

\section{Introduction}
\label{sec1}

The subject of the present paper is a fractional Brownian motion (fBm) \index{fractional Brownian motion}
$B^{H}=\{B_{t}^{H}, t\ge 0\}$ with the Hurst index
$H\in  \left(0, \, \frac{1}{2}\right)$.\index{Hurst index} Generally speaking, a fBm with the Hurst
index $H\in (0,1)$\index{Hurst index} is a Gaussian process with zero mean and the covariance
function of the form
\begin{equation*}
{\mathrm{E}}B_{t}^{H} B_{s}^{H} =\frac{1}{2}(t^{2H}+s^{2H}-\vert t-s\vert ^{2H}).
\end{equation*}
Its properties are rather different for
$H\in \left (0, \, \frac{1}{2}\right )$ and
$H\in \left (\frac{1}{2}, \, 1\right )$. In particular,
$H\in \left (0, \, \frac{1}{2}\right )$ implies short-term dependence.
In contrast, $H\in \left (\frac{1}{2}, \, 1\right )$ implies long-term
dependence. Moreover, technically it is easier to deal with fBms having
$H\in \left (\frac{1}{2}, \, 1\right )$. Due to this and many other reasons,
fBm with $H\in \left (\frac{1}{2}, \, 1\right )$ has been much more intensively
studied in the recent years. However, the financial markets in which trading
takes place quite often, demonstrate the presence of a short memory, and
therefore the volatility in such markets (so called rough volatility) is
well modeled by fBm with $H\in \left (0, \, \frac{1}{2}\right )$, see e.g.
\cite{Volatility}. Thus interest to fBm with small Hurst indices\index{Hurst index} has substantially
increased recently. Furthermore, it is well known that a fractional Brownian
motion\index{fractional Brownian motion} is neither a Markov process nor semimartingale, and especially it
is neither martingale nor a process with independent increments\index{independent increments} unless
$H=\frac{1}{2}$. That is why it is naturally to search the possibility
of the approximation of fBm in a certain metric by simpler processes, such
as Markov processes, martingales, semimartingales or a processes of bounded
variation. As for the processes of bounded variation and semimartingales,
corresponding results are presented in \cite{Andr05,Andr06} and
\cite{Thao03}. In the papers
\cite{BM08,BMishura10,MishuraB08,5stars} approximation of a fractional
Brownian motion\index{fractional Brownian motion} with Gaussian martingales\index{Gaussian martingales} was studied and summarized in
the monograph \cite{MonoBMRS19}, but most of problems were considered only
for $H\in \left (\frac{1}{2}, \, 1\right )$, for the reasons stated above.

In the present paper we continue to consider the approximation of a fractional
Brownian motion\index{fractional Brownian motion} by Gaussian martingales\index{Gaussian martingales} but concentrate on the case
$H\in \left (0, \frac{1}{2}\right )$.

Let $(\Omega , \mathcal{F}, \mathrm{P})$ be a complete probability space
with a filtration $\mathbb{F} = \{ \mathcal{F}_{t}\}_{t\geq 0}$ satisfying
the standard assumptions. We start with the Molchan representation\index{Molchan representation} of fBm
via the Wiener process\index{Wiener process} on a finite interval. Namely, it was established
in \cite{Norros} that the fBm
$\{B_{t}^{H} , \mathcal{F}_{t}, \,t\geq 0\}$ can be represented as
%
\begin{equation}
\label{BtHK}
B_{t}^{H} =\int \limits _{0}^{t} {z(t,s)} dW_{s},
\end{equation}
where $\{W_{t} ,\,\,t\in [0,\,T]\}$ is a Wiener process,\index{Wiener process}
\begin{equation*}
\begin{gathered}
z(t,s)=c_{H} \bigg (t^{H-1/2} s^{1/2-H}(t-s)^{H-1/2}
\\
-(H-1/2) s^{1/2-H} \int \limits _{s}^{t} u^{H-3/2}(u-s)^{H-1/2}du
\bigg ),
\end{gathered}
\end{equation*}
is the Molchan kernel,\index{Molchan kernel}
%
\begin{equation}
\label{cH}
c_{H} =\left (
\frac{2H\cdot \Gamma (\frac{3}{2}-H)}{\Gamma (H+\frac{1}{2})\cdot \Gamma (2-2H)}
\right )^{1/2},
\end{equation}
and $\Gamma (x)$, $x>0$, is the Gamma function.\index{Gamma function}

Let us consider a problem of the distance between a fractional Brownian motion\index{fractional Brownian motion}
and the space of square integrable martingales (initially not obligatory
Gaussian), adapted to the same filtration. So, we are looking for a square
integrable $\mathbb{F}$-martingale $M$ with the bracket that is absolutely
continuous w.r.t. (with respect to) the Lebesgue measure such that it minimizes
the value
\begin{equation*}
\rho _{H}(M)^{2}:=\sup _{t\in [0,T]} {\mathrm{E}}(B_{t}^{H}-M_{t})^{2}.
\end{equation*}

We observe first that $B^{H}$ and $W$ generate the same filtration, so
any square integrable $\mathbb{F}$-martingale $M$ with the bracket that
is absolutely continuous w.r.t. the Lebesgue measure, admits a representation
%
\begin{equation}
\label{itorep}
M_{t} = \int _{0}^{t} a(s) dW_{s},
\end{equation}
where $a$ is an $\mathbb{F}$-adapted square integrable process such that
$\langle M\rangle _{t}=\int _{0}^{t}a^{2}(s)ds$. Hence we can write, see
\cite{5stars},
\begin{align*}
{\mathrm{E}}(B_{t}^{H} - M_{t})^{2} &= {\mathrm{E}}\left (\int _{0}^{t} (z(t,s)-a(s))
dW_{s}\right )^{2} = \int _{0}^{t} {\mathrm{E}}(z(t,s)-a(s))^{2} ds
\\
&= \int _{0}^{t} (z(t,s)- {\mathrm{E}} a(s))^{2} ds + \int _{0}^{t}
\mathrm{Var}\; a(s) ds.
\end{align*}
Consequently, it is enough to minimize $\rho _{H}(M)$ over continuous
\emph{Gaussian} martingales.\index{continuous Gaussian martingales} Such martingales have orthogonal and therefore
independent increments.\index{independent increments} Then the fact that they have representation \eqref{itorep} with a non-random $a$ follows, e.g., from
\cite{Skorohod}.

Now let $a:[0,T]\to \mathbb{R}$ be a nonrandom measurable function of the
class $L_{2} [0,T]$; that is, $a$ is such that the stochastic integral
$\int \limits _{0}^{t} {a(s)dW_{s} } $, $t\in [0,\;T]$, is well defined
w.r.t. the Wiener process $\{W_{t} ,\;\;t\in [0,\;T]\}$\index{Wiener process} (this integral
is usually called the Wiener integral if the integrand is a nonrandom function).
So, the problem is to find
\begin{equation*}
\mathop{\inf }\limits _{a\in L_{2} [0,T]} \mathop{\sup }\limits _{0
\le t\le T} {\mathrm{E}}\left (B_{t}^{H} -\int \limits _{0}^{t} {a(s)dW_{s}
} \right )^{2}= \mathop{\inf }\limits _{a\in L_{2} [0,T]} \mathop{
\sup }\limits _{0\le t\le T} \int \limits _{0}^{t} (z(t,s)- a(s) )^{2}
ds .
\end{equation*}

Note that the expression
to be minimized does not involve neither the fractional
Brownian motion\index{fractional Brownian motion} nor the Wiener process,\index{Wiener process} so the problem becomes purely analytic.
Moreover, since the problem posed in a general form is not observable and
accessible for solution, we restrict ourselves to one particular subclass
of functions. We study the class
\begin{equation*}
\{a(s)=s^{\gamma }, \gamma >0\}.
\end{equation*}
Our main result is Theorem \ref{thdistpowf}, which shows where
$\mathop{\max }\limits _{t\in [0, 1]} {\mathrm{E}}\left (B^{H}_{t} - \int _{0}^{t}
a(s) dW_{s} \right )^{2}$ could be reached, depending on $\gamma >0$. We
also provide remarks after the theorem.

\section{Distance from fBm with $H\in (0, 1/2)$ to the subspaces of Gaussian martingales\index{Gaussian martingales} involving power integrands}
\label{sec2}

Consider a class of power functions with an arbitrary positive exponent.
Thus, we now introduce the class
\begin{equation*}
\{a(s)=s^{\gamma }, \gamma >0\}.
\end{equation*}
For the sake of simplicity, let $T = 1$.

\begin{thm}%
\label{thdistpowf}
Let $a=a(s)$ be a function of the form $a(s)=s^{\gamma }$, $\gamma >0$,
$H \in (0, 1/2)$. Then:
\begin{itemize}
\item[\upshape(\textit{i})] For all $\gamma >0$ the maximum
$\mathop{\max }\limits _{t\in [0, 1]} {\mathrm{E}}\left (B^{H}_{t} - \int _{0}^{t}
s^{\gamma }dW_{s} \right )^{2}$ is reached at one of the following points:
$t=1$ or $t=t_{1}$, where
\begin{align*}
t_{1} &= \Bigg (\;c_{H} B\left (\gamma -H+\frac{3}{2}, H+\frac{1}{2}
\right )(\gamma +1)
\\
&-\sqrt{c_{H}^{2}\left (B\left (\gamma -H+\frac{3}{2}, H+\frac{1}{2}
\right )(\gamma +1)\right )^{2}-2H}\;\; \Bigg )^{
\frac{1}{\gamma -H+\frac{1}{2}}}.
\end{align*}
\item[\upshape(\textit{ii})] For any $H\in (0,1/2)$ there exists
$\gamma _{0}=\gamma _{0}(H)> 0$ such that for $\gamma >\gamma _{0}$ the maximum
\begin{align*}
&\mathop{\max }\limits _{t\in [0, 1]} {\mathrm{E}}\left (B^{H}_{t} - \int _{0}^{t}
s^{\gamma }dW_{s} \right )^{2}
\\
&= t_{1}^{2H}-2 t_{1}^{\gamma +\frac{1}{2}+H} c_{H} B\left (\gamma -H+
\frac{3}{2},H+\frac{1}{2}\right ) \frac{\gamma +1}{\gamma +\frac{1}{2}+H}+
\frac{1}{2\gamma +1} t_{1}^{2\gamma +1}
\end{align*}
and is reached at the point $t_{1}$. Here $B(x,y)$, $x,y>0$, is a beta function.
\end{itemize}

\end{thm}

\begin{proof}
According to Lemma 2.20 \cite{MonoBMRS19}, the distance between the fractional
Brownian motion\index{fractional Brownian motion} and the integral $\int _{0}^{t}s^{\gamma }dW_{s}$ w.r.t.
Wiener process\index{Wiener process} $t\in [0, 1]$ equals
%
\begin{align}
E &\left (B^{H}_{t} - \int _{0}^{t} s^{\gamma }dW_{s}\right )^{2} = E
\left (B^{H}_{t} \right )^{2}-2E\left (\int _{0}^{t} z(t,s) dW_{s}
\int _{0}^{t} s^{\gamma }dW_{s} \right )
\nonumber
\\
&+E\left (\int _{0}^{t} s^{\gamma }dW_{s}\right )^{2} = t^{2H}-2\int _{0}^{t}
z(t,s)s^{\gamma } ds + \int _{0}^{t} s^{2\gamma } ds
\nonumber
\\
&=t^{2H} - 2t^{\gamma +H+\frac{1}{2}}c_{H}B\left (\gamma -H+
\frac{3}{2},H+\frac{1}{2}\right )
\frac{\gamma +1}{\gamma +H+\frac{1}{2}}
\nonumber
\\
&+\frac{t^{2\gamma +1}}{2\gamma +1} := h(t,\gamma ),
\label{htgam}
\end{align}
where $c_{H}$ is taken from \eqref{cH}.\vadjust{\goodbreak}

Let us calculate the partial derivative of $h(t,\gamma )$ in $t$:
\begin{align*}
\frac{\partial h(t,\gamma )}{\partial t} = t^{2H-1}&\Big (2H-2t^{
\gamma -H+\frac{1}{2}}c_{H}
\\
&\cdot B\left (\gamma -H+\frac{3}{2},H+\frac{1}{2}\right )(\gamma +1)+t^{2(
\gamma -H+\frac{1}{2})}\Big ).
\end{align*}

Let us verify whether there is $t\in [0, 1]$ such that
$\frac{\partial h(t,\gamma )}{\partial t}=0$, i.e.
\begin{equation*}
t^{2(\gamma -H+\frac{1}{2})} -2t^{\gamma -H+\frac{1}{2}}c_{H} B\left (
\gamma -H+\frac{3}{2},H+\frac{1}{2}\right )(\gamma +1)+ 2H=0.
\end{equation*}

Changing the variable $t^{\gamma -H+\frac{1}{2}}=: x$, we obtain the following
quadratic equation:
%
\begin{equation}
\label{eqquadrat}
x^{2}-2xc_{H} B\left (\gamma -H+\frac{3}{2},H+\frac{1}{2}\right )(
\gamma +1)+2H=0.
\end{equation}

The discriminant $D=D(\gamma )$ of the quadratic equation \eqref{eqquadrat} equals
%
\begin{align}
D(\gamma ) &=4c_{H}^{2}\left (B\left (\gamma -H+\frac{3}{2},H+
\frac{1}{2}\right )(\gamma +1)\right )^{2}-8H
\nonumber
\\
&= 8H\left (\left (B\left (\gamma -H+\frac{3}{2},H+\frac{1}{2}\right )(
\gamma +1)\right )^{2}
\frac{\Gamma (\frac{3}{2}-H)}{\Gamma (H+\frac{1}{2})\Gamma (2-2H)}-1
\right )
\nonumber
\\
&= 8H\left (
\frac{\Gamma (H+\frac{1}{2})\Gamma (\frac{3}{2}-H)}{\Gamma (2-2H)}
\left (\frac{\Gamma (\gamma -H+\frac{3}{2})}{\Gamma (\gamma +1)}
\right )^{2}-1\right ).
\label{eqdiscr}
\end{align}

Now we are going to show that $D(0)>0$ and $D(\gamma )$ is increasing in $\gamma >0$. For this we study separately the function
$ f(H):=
\frac{\Gamma (H+\frac{1}{2})(\Gamma (\frac{3}{2}-H))^{3}}{\Gamma (2-2H)}$
for $H \in (0,\frac{1}{2})$.

\begin{figure}[h]
\includegraphics{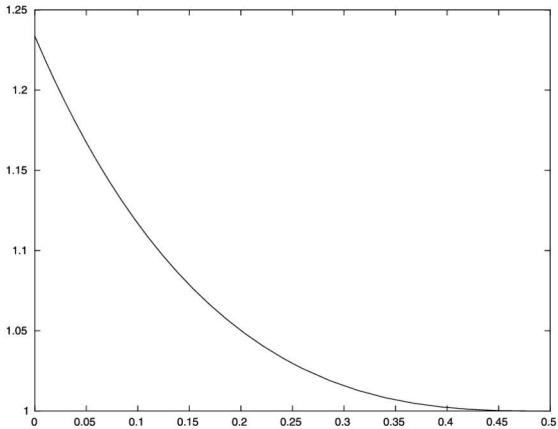}
\caption{The behavior of $f(H)$ for $H \in (0,\frac{1}{2})$}
\label{fig:fH}
\end{figure}
\begin{figure}[h]
\includegraphics{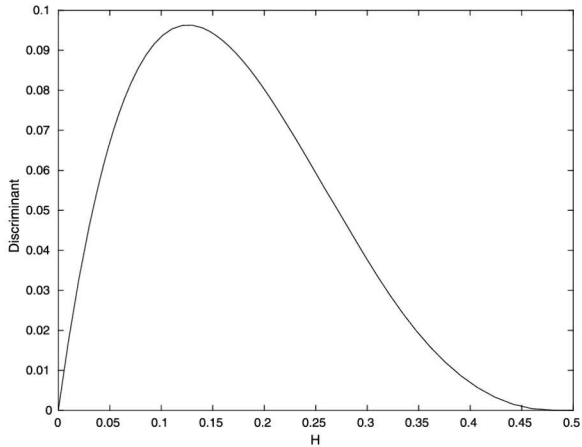}
\caption{The behavior of $D(0)$ as a function of $H \in (0,\frac{1}{2})$}
\label{fig:D0}
\end{figure}

Let us calculate the values of this function at the following points:
$f(0)=\frac{\pi ^{2}}{8}>1$, $\smallskip f(\frac{1}{2})=1$. To establish that
$f(H)$ is decreasing in $H$, consider Lemma \ref{lem:gamder} and the following calculation:
%
\begin{align}
(\ln f(H))_{H}' &=\left (\ln
\frac{\Gamma (H+\frac{1}{2})(\Gamma (\frac{3}{2}-H))^{3}}{\Gamma (2-2H)}
\right )_{H}'
\nonumber
\\
&= \left (\ln \Gamma \left (H+\frac{1}{2}\right )+3\ln \Gamma \left (
\frac{3}{2}-H\right )-\ln \Gamma \left (2-2H\right )\right )_{H}'
\nonumber
\\
&= \frac{\Gamma '(H+\frac{1}{2})}{\Gamma (H+\frac{1}{2})}-3
\frac{\Gamma '(\frac{3}{2}-H)}{\Gamma (\frac{3}{2}-H)}+2
\frac{\Gamma '(2-2H)}{\Gamma (2-2H)}
\nonumber
\\
&= \int _{0}^{1} \frac{1-t^{H-\frac{1}{2}}}{1-t} dt-C-3\left (\int _{0}^{1}
\frac{1-t^{\frac{1}{2}-H}}{1-t} dt-C\right )
\nonumber
\\
&\qquad +2\left (\int _{0}^{1} \frac{1-t^{1-2H}}{1-t} dt-C\right )
\nonumber
\\
&= \int _{0}^{1}
\frac{1-t^{H-\frac{1}{2}}-3+3t^{\frac{1}{2}-H}+2-2t^{1-2H}}{1-t} dt
\nonumber
\\
&= \int _{0}^{1} \frac{3t^{\frac{1}{2}-H}-2t^{1-2H}-t^{H-\frac{1}{2}}}{1-t} dt
\nonumber
\\
&=\int _{0}^{1}
\frac{t^{H-\frac{1}{2}}(3t^{1-2H}-2t^{\frac{3}{2}-3H}-1)}{1-t} dt.
\label{eqlnfh}
\end{align}

Let $t \in (0,1)$. Obviously, in this case $t^{H- \frac{1}{2}}>0$ and
$ 1-t >0$. Changing the variables in \eqref{eqlnfh} as
$z:=t^{\frac{1}{2}-H}$, we get
\begin{equation*}
3z^{2}-2z^{3}-1= -(1-z)^{2}(2z+1),
\end{equation*}
and this function is negative for all $z \in (0,1)$. Hence,
$(\ln f(H))'<0 $ and it means that $f(H)$ is decreasing. Furthermore,
$ f(H)>1$ for every $H \in (0,\frac{1}{2})$. The behavior of $ f(H)$ is presented
in Figure \ref{fig:fH}.

So, we proved that $D(0)>0$ (the behavior of $D(0)$ as a function of
$H$ is presented in Figure \ref{fig:D0}), and it follows from Lemma~\ref{lem:z-gam}
that $D(\gamma )$ is increasing in $\gamma >0$ for any
$H\in (0,1/2)$. Therefore, for every $H \in (0,\frac{1}{2})$ and
$\gamma >0$ we have that the quadratic equation \eqref{eqquadrat} has two
roots.

More precisely, if you use standard notations for the coefficients of the
quadratic equation, then coefficient $a$ at $x^{2}$ in \eqref{eqquadrat} is strongly positive, $a=1$, coefficient at $x$ equals
$b=-2c_{H}B(\gamma -H+\frac{3}{2},H+\frac{1}{2})(\gamma +1)$ and is negative,
and $c=2H>0$. We conclude that our quadratic equation has two positive
roots $x_{1}=\frac{-b-\sqrt{b^{2}-4ac}}{2a}$ and
$x_{2}=\frac{-b+\sqrt{b^{2}-4ac}}{2a}$, $x_{1}\le x_{2}$.

According to our notations, we let
$t_{i} :=x_{i}^{\frac{1}{\gamma -H+\frac{1}{2}}}$, $i=1,2$. Since
$x=t^{\gamma -H+\frac{1}{2}}\in [0,1]$ for $t \in [0,1]$ and the left-hand side
of \eqref{eqquadrat} is negative for $x \in (x_{1},x_{2})$, we get the
following cases:
\begin{itemize}
\item[(\textit{i})] Let $x_{1}<1$ and $x_{2}<1$. Then
$\max \limits _{t \in [0,1]}h(t,\gamma )$ can be achieved at one of two
points: $t=t_{1}$ or $t= 1$.
\item[(\textit{ii})] Let $x_{1}<1$ and $x_{2}\ge 1$. Then
$\max \limits _{t \in [0,1]}h(t,\gamma )$ is achieved at point
$t=t_{1}$.
\item[(\textit{iii})] Let $x_{1}\ge 1$ (and consequently $x_{2}>1$). Then
$\max \limits _{t \in [0,1]}h(t,\gamma )$ is achieved at point
$t= 1$.
\end{itemize}
Now, we rewrite the discriminant \eqref{eqdiscr} in the following form:
%
\begin{align}
D &=4\left (c_{H}^{2}\left (B\left (\gamma -H+\frac{3}{2},H+
\frac{1}{2}\right )(\gamma +1)\right )^{2}-2H\right )
\nonumber
\\
&=: 4(d_{H}^{2}(\gamma )-2H)>0,
\label{eq:Dtr}
\end{align}
where
$d_{H}(\gamma )=c_{H}B(\gamma -H+\frac{3}{2},H+\frac{1}{2})(\gamma +1)$.
From Lemma \ref{lem:d_H}, $x_{1}<\sqrt{2H}<1$, so case (\textit{iii}) never occurs.

According to Lemma \ref{lem:d_H}, the biggest of two roots,
$x_{2}=d_{H}(\gamma )+\sqrt{d_{H}^{2}(\gamma )-2H}$, is increasing in
$\gamma >0$ and $x_{2}>1 \Leftrightarrow d_{H}>\frac{1}{2}+H$. Moreover, it
follows from Lemma \ref{lem:d_H}, (\textit{iv}) that
$d_{H}(\gamma ) \to +\infty $ as $\gamma \to +\infty $. Therefore, for
all $H \in (0,\frac{1}{2})$ there exists $\gamma _{0}(H)>0$ such that for all
$\gamma >\gamma _{0}$ we have $x_{2}>1$.

It means that for $\gamma >\gamma _{0}$ our maximum is reached at the point
$t_{1} = (x_{1})^{\frac{1}{\gamma +\frac{1}{2}-H}}$. Finally,
\begin{align*}
&\mathop{\max }\limits _{t\in [0, 1]} {\mathrm{E}}\left (B^{H}_{t} - \int _{0}^{t}
s^{\gamma }dW_{s} \right )^{2}
\\
&= t_{1}^{2H}-2 t_{1}^{\gamma +\frac{1}{2}+H} c_{H} B\left (\gamma -H+
\frac{3}{2},H+\frac{1}{2}\right ) \frac{\gamma +1}{\gamma +\frac{1}{2}+H}+
\frac{1}{2\gamma +1} t_{1}^{2\gamma +1},
\end{align*}
where $t_{1} = (x_{1})^{\frac{1}{\gamma +\frac{1}{2}-H}}$.
\end{proof}

\begin{remark}%
\label{rem:gam0curve}%
The implicit equation $d_{H}(\gamma )=\frac{1}{2}+H$ considered as the equation
for $\gamma _{0}$ as a function of $H$ gives us the relation between
$H \in \left (0,\frac{1}{2}\right )$ and respective $\gamma _{0}>0$ which,
by virtue of the foregoing, is determined unambiguously. The form of the
algebraic curve $\gamma _{0}=\gamma _{0}(H)$\index{algebraic curve} is presented in Figure \ref{fig:DH}.
%
\begin{figure}[h]
\includegraphics{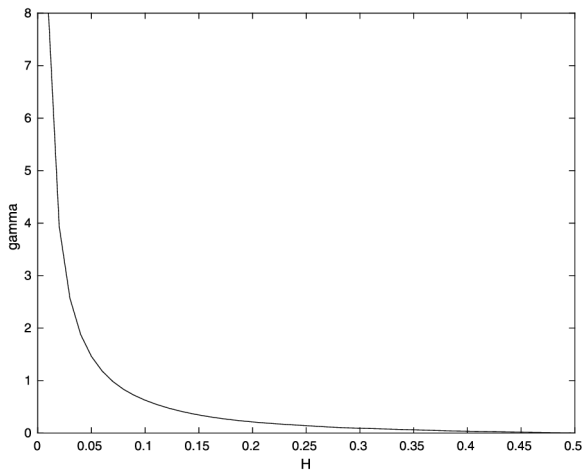}
\caption{The algebraic curve $\gamma _{0}=\gamma _{0}(H)$}
\label{fig:DH}
\end{figure}
\end{remark}

Consider one of the coefficients that are present in \eqref{htgam},
namely,
\begin{align*}
c^{1}_{H,\gamma } &=2c_{H}B\left (\gamma -H+\frac{3}{2},H+\frac{1}{2}\right )
\frac{\gamma +1}{\gamma +H+\frac{1}{2}}
\\
&= 2\left (
\frac{2H\Gamma \left (\frac{3}{2}-H\right )}{\Gamma \left (H+\frac{1}{2}\right )\Gamma \left (2-2H\right )}
\right )^{1/2}
\frac{\Gamma \left (\gamma -H+\frac{3}{2}\right )\Gamma \left (H+\frac{1}{2}\right )}{\Gamma \left (\gamma +1\right )\left (\gamma +H+\frac{1}{2}\right )}.
\end{align*}
Let $\gamma =0$. Then
\begin{align*}
c^{1}_{H,0} &=2\left (
\frac{2H\Gamma \left (\frac{3}{2}-H\right )}{\Gamma \left (H+\frac{1}{2}\right )\Gamma \left (2-2H\right )}
\right )^{1/2}
\frac{\Gamma \left (\frac{3}{2}-H\right )\Gamma \left (H+\frac{1}{2}\right )}{H+\frac{1}{2}}
\\
&= \frac{2}{H+\frac{1}{2}}\left (
\frac{2H\Gamma \left (\frac{3}{2}-H\right )}{\Gamma \left (H+\frac{1}{2}\right )\Gamma \left (2-2H\right )}
\right )^{1/2}
\frac{\pi \left (\frac{1}{2}-H\right )}{\sin \left (\pi \left (\frac{1}{2}-H\right )\right )}.
\end{align*}
For $H=\frac{1}{2}$, one has $c_{1/2,0}^{1}=2$. Obviously,
$c_{H,0}^{1} \rightarrow 0$ as $H \downarrow 0$. It means that
$c_{H,0}^{1}$ is small in some neighborhood of zero. Having this in mind,
we establish some sufficient condition for $h(t,0)$ to get its maximum
at point 1.

\begin{lemma}%
\label{lem:sufmaxat1}
Let $c_{H,0}^{1}<1$. Then $h(t_{1},0)<h(1,0)$.
\end{lemma}
\begin{proof}
Indeed, $h(t_{1},0)=t_{1}^{2H}-c_{H,0}^{1}t_{1}^{1/2+H}+t_{1}$, while
$h(1,0)=2-c_{H,0}^{1}$. Then the inequality $h(t_{1},0)<h(1,0)$ is equivalent
to the following one:
\begin{align*}
c_{H,0}^{1}\left (1-t_{1}^{1/2+H}\right ) &< 2-t_{1}-t_{1}^{2H}.
\end{align*}
If $c_{H,0}^{1}<1$, then
$c_{H,0}^{1} \bigl(1-t_{1}^{1/2+H} \bigr)<1-t_{1}^{1/2+H} < 1-t_{1} <
2-t_{1}-t_{1}^{2H}$.
\end{proof}

\begin{remark}%
\label{rem:tosufmax}
As it was mentioned before, $c_{1/2,0}^{1}=2$, and so for
$H = \frac{1}{2}$ the condition of Lemma \ref{lem:sufmaxat1} is not satisfied.
In this case $d_{1/2}(0)=1$, and $t_{1}=t_{2}=1$, so that we have the equality
$h(t_{1},0)=h(1,0)$.

However, the question what will be for $\gamma =0$ and
$H_{0}<H<\frac{1}{2}$, where $H_{0}$ is such a value that for
$0<H<H_{0}$, $c_{H,0}^{1}<1$, is open. In order to fill this gap, we provide
the numerical results with some comments.

Consider two fuctions $h(t_{1},\gamma )$ and $h(1,\gamma )$ as functions
of $\gamma $ and $H$. We already know that
$\mathop{\max }\limits _{t\in [0, 1]} h(t,\gamma ) = \max \{h(t_{1},
\gamma ),h(1,\gamma )\}$. The projection of the surface of
$\max \{h(t_{1},\gamma ),h(1,\gamma )\}$ on the $(H,\gamma )$-plane is presented
in Figure \ref{fig:surfaceprojection}. Points, where $h(t,\gamma )$ reaches
its maximum at $t=1$ are represented in green color, and points where
$h(t,\gamma )$ reaches its maximum at $t=t_{1}$ are represented in brown
color. The black curve is  the algebraic curve
$\gamma _{0}=\gamma _{0}(H)$, which is presented in Figure \ref{fig:DH}.

\begin{figure}[h]
\includegraphics{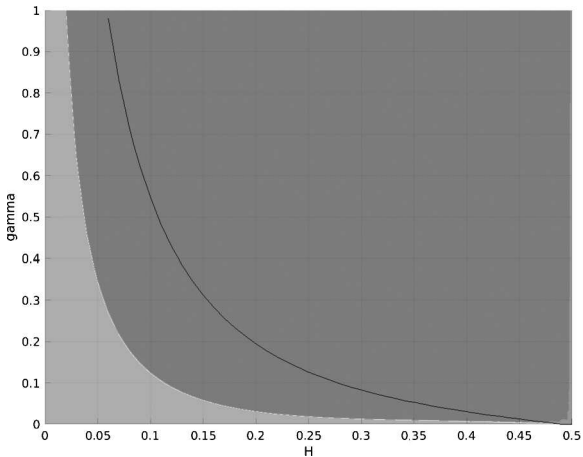}
\caption{The projection of the surface of
$\mathop{\max }\limits _{t\in [0, 1]} h(t,\gamma )$}
\label{fig:surfaceprojection}
\end{figure}
\end{remark}


\begin{appendix}
\section*{Appendix section}
\label{secA}

In the proof of Theorem \ref{thdistpowf}, we use these auxiliary results.

\begin{lemma}%
\label{lem:srtinc}
Let $f:\mathbb{R} \rightarrow \mathbb{R}$ be a strictly convex function
of one variable. Take the function
$g(x_{1}, x_{2}) = \frac{f(x_{1})-f(x_{2})}{x_{1}-x_{2}}$, $x_{1} \neq x_{2}$, $x_{1},x_{2} \in \mathbb{R}$. Then $g(x_{1}+\alpha ,x_{2}+\alpha )$ is strictly
increasing in $\alpha > 0$.
\end{lemma}

\begin{lemma}%
\label{lem:gamder}
Let $\Gamma (x)$ be the Gamma function.\index{Gamma function} Then
\begin{equation*}
(\ln \Gamma (x))' = \frac{\Gamma '(x)}{\Gamma (x)}=\int _{0}^{1}
\frac{1-t^{x-1}}{1-t} dt-C,
\end{equation*}
where $C$ is a fixed constant.
\end{lemma}

Proofs of Lemma 2 and Lemma 3 could be found in \cite{Fichtenholz}.

\begin{lemma}%
\label{lem:z-gam}
If $H \in \left (0,\frac{1}{2}\right )$, then the function
$z(\gamma ):=
\frac{\Gamma  (\gamma -H+\frac{3}{2} )}{\Gamma (\gamma +1)}$ is
increasing in $\gamma >0$.
\end{lemma}

\begin{proof}
For every $\gamma >0$ we have $z(\gamma )>0$. Let us calculate
\begin{align*}
\ln z(\gamma )&=\ln
\frac{\Gamma \left (\gamma -H+\frac{3}{2}\right )}{\Gamma (\gamma +1)}=
\\
&= \left (\frac{1}{2}-H\right )
\frac{\ln \Gamma \left (\gamma -H+\frac{3}{2}\right ) -
\ln \Gamma (\gamma +1)}{\left (\gamma -H+\frac{3}{2}\right )-(\gamma +1)}=:
\left (\frac{1}{2}-H\right )\omega (\gamma ).
\end{align*}
According to Lemma \ref{lem:srtinc} and the fact that
$\ln (\Gamma (x))$ is strictly convex we have that
$\omega (\gamma )$ is increasing. Since
$\left (\frac{1}{2} - H\right )>0$, it is clear that $z(\gamma )$ is increasing
in $\gamma >0 $.
\end{proof}

\begin{lemma}%
\label{lem:d_H}
Let
$d_{H}(\gamma )=c_{H}B(\gamma -H+\frac{3}{2},H+\frac{1}{2})(\gamma +1)$,
and
$x_{1}=d_{H}(\gamma )-\sqrt{d_{H}^{2}(\gamma )-2H}$, $x_{2}=d_{H}(
\gamma )+\sqrt{d_{H}^{2}(\gamma )-2H}$ be roots of the quadratic equation \eqref{eqquadrat}.
Then for all $\gamma >0$, $H\in \bigl(0,\frac{1}{2}\bigr)$ the following
statements hold.
\begin{itemize}
\item[\upshape\textit{i})] $d_{H}(\gamma )$ is increasing in $\gamma >0$.
\item[\upshape\textit{ii})] $x_{1}<\sqrt{2H}$ and $x_{2}>\sqrt{2H}$.
\item[\upshape\textit{iii})] $d_{H}(\gamma )>\frac{1}{2}+H \Leftrightarrow x_{2}>1$ and
$x_{1}<2H$.
\item[\upshape\textit{iv})] $d_{H}(\gamma )\rightarrow \infty $ as
$\gamma \rightarrow \infty $.
\end{itemize}
\end{lemma}

\begin{proof}
(\textit{i}) Note that $d_{H}(\gamma )$ is increasing in $\gamma >0$ since
\begin{equation*}
d_{H}(\gamma )=c_{H} \Gamma \left (H+\frac{1}{2}\right )
\frac{\Gamma \left (\gamma -H+\frac{3}{2}\right )}{\Gamma \left (\gamma +1\right )},
\end{equation*}
where $c_{H} \Gamma \left (H+\frac{1}{2}\right )>0$ for all
$H \in \left (0,\frac{1}{2}\right )$, and according to Lemma \ref{lem:z-gam},
$
\frac{\Gamma  (\gamma -H+\frac{3}{2} )}{\Gamma  (\gamma +1 )}$
is increasing in $\gamma >0$.

(\textit{ii}) Discriminant \eqref{eqdiscr} satisfies the following relation:
\begin{align*}
0<D\left (0\right ) &= 4\left (d_{H}^{2}\left (0\right )-2H\right ).
\end{align*}
Therefore, $d_{H}(\gamma )>\sqrt{2H}$ for all $\gamma >0$. Also, we can
rewrite $x_{1}=d_{H}(\gamma )-\sqrt{d_{H}^{2}(\gamma )-2H}$ and
$x_{2}=d_{H}(\gamma )+\sqrt{d_{H}^{2}(\gamma )-2H}$. Hence
$x_{2}>\sqrt{2H}$. Transform the value $x_{1}$ to the following form:
\begin{equation*}
x_{1}=
\frac{d_{H}^{2}(\gamma )-(d_{H}^{2}(\gamma )-2H)}{d_{H}(\gamma )+\sqrt{d_{H}^{2}(\gamma )-2H}}=
\frac{2H}{x_{2}}<\sqrt{2H}.
\end{equation*}

(\textit{iii}) Let us assume that $x_{2}>1$ (or, what is equivalent,
$x_{1}=\frac{2H}{x_{2}}<\break 2H$). In turn, this is equivalent to the relation
$d_{H}(\gamma )+\sqrt{d_{H}^{2}(\gamma )-2H}>1$, or
$\sqrt{d_{H}^{2}(\gamma )-2H}>1-d_{H}(\gamma )$. The latter inequality
can be realized in one of two cases:
%
\begin{equation}
\label{eqmox1}
d_{H}^{2}(\gamma )-2H>1-2d_{H}(\gamma )+d_{H}^{2}(\gamma ),\quad 1-d_{H}(
\gamma )\ge 0;
\end{equation}
or
%
\begin{equation}
\label{eqmox2}
1-d_{H}(\gamma )<0.
\end{equation}

The couple of inequalities \eqref{eqmox1} is equivalent to
$ 1\ge d_{H}(\gamma )>\frac{1}{2}+H$. Therefore, inequalities \eqref{eqmox1} and \eqref{eqmox2}, taken together, indicate that
$d_{H}(\gamma )>\frac{1}{2}+H$ if and only if $x_{2}>1$.

(\textit{iv}) The value $d_{H}(\gamma )$ can be presented as
\begin{equation*}
d_{H}(\gamma )=c_{H}
\frac{\Gamma \left (\gamma -H+\frac{3}{2}\right )\Gamma \left (H+\frac{1}{2}\right )}{\Gamma \left (\gamma +2\right )}
\left (\gamma +1\right ),
\end{equation*}
where $c_{H}\Gamma \left (H+\frac{1}{2}\right )>0$ is a fixed constant, and for
all $H \in \left (0,\frac{1}{2}\right )$, 
\begin{align*}
\frac{\Gamma \left (\gamma -H+\frac{3}{2}\right )}{\Gamma \left (\gamma +2\right )}
\left (\gamma +1\right )&=
\frac{\Gamma \left (\gamma -H+\frac{3}{2}\right )}{\Gamma \left (\gamma +1\right )}
\\
&=
\frac{\sqrt{\frac{2\pi }{\gamma -H+\frac{3}{2}}}\left (\frac{\gamma -H+\frac{3}{2}}{e}\right )^{\gamma -H+\frac{3}{2}}\left (1+O\left (\frac{1}{\gamma -H+\frac{3}{2}}\right )\right )}{\sqrt{\frac{2\pi }{\gamma +1}}\left (\frac{\gamma +1}{e}\right )^{\gamma +1}\left (1+O\left (\frac{1}{\gamma +1}\right )\right )}
\\
&\sim \frac{1}{e^{\frac{1}{2}-H}}
\frac{\left (\gamma -H+\frac{3}{2}\right )^{\gamma -H+1}}{\left (\gamma +1\right )^{\gamma +\frac{1}{2}}}
\rightarrow \infty , \quad \gamma \rightarrow \infty ,
\end{align*}
%
which follows from the Stirling's approximation for Gamma function.\index{Gamma function}
\end{proof}
\end{appendix}



\end{document}